\documentclass{amsart}

\usepackage{tikz}
\usepackage{titlesec}
\usepackage{amsmath, amssymb, amsfonts, amsthm}
\usepackage{asymptote}
\usepackage{enumerate}
\usepackage{versions}
\usepackage[utf8]{inputenc}
\usepackage{fancyhdr}
\usepackage{etoolbox}
\usepackage{extarrows}
\usepackage{mathtools}
\usepackage[colorlinks]{hyperref}
\usepackage{relsize}
\usepackage{thm-restate}

\newcommand{\ZZ}{\mathbb Z}
\newcommand{\NN}{\mathbb N}

\newtheorem{theorem}{Theorem}[section]
\newtheorem{lemma}[theorem]{Lemma}
\newtheorem{corollary}[theorem]{Corollary}
\newtheorem{proposition}[theorem]{Proposition}
\newtheorem{conjecture}[theorem]{Conjecture}

\theoremstyle{definition}

\usepackage{titlesec}
\usepackage[foot]{amsaddr}

\titleformat{name=\section}{\centering\scshape}{\thetitle.}{0.5em}{}
\titleformat{name=\subsection}[runin]{\itshape}{\thetitle.}{0.5em}{}[.]
\titleformat{name=\subsubsection}[runin]{}{\thetitle.}{0.5em}{\itshape}[.]

\newtheorem{prob}[theorem]{Problem}

\newcommand{\Zmn}{\ZZ/m\ZZ \times \ZZ/n\ZZ}
\newcommand{\Znn}{(\ZZ/n\ZZ)^2}
\newcommand{\Znd}{(\ZZ/n\ZZ)^d}

\title{Modified Erd\"os--Ginzburg--Ziv constants for $\ZZ/n\ZZ$ and $\Znn$}

\author[Aaron Berger]{Aaron Berger}
\author[Danielle Wang]{Danielle Wang}
\address{Department of Mathematics, Massachusetts Institute of Technology, Cambridge, MA 02139-4307}
\email{bergera@mit.edu}
\email{diwang@mit.edu}

\begin{document}
\begin{abstract}
	For an abelian group $G$ and an integer $t > 0$, the \emph{modified Erd\"os--Ginzburg--Ziv constant}
	$s_t'(G)$ is the smallest integer $\ell$ such that
	any zero-sum sequence of length at least $\ell$ with
	elements in $G$ contains a zero-sum subsequence (not necessarily consecutive)
	of length $t$. We compute $s_t'(G)$ for
	$G = \ZZ/n\ZZ$ and for $t = n$, $G = \Znn$.
\end{abstract}

\maketitle

\section{Introduction}
In 1961, Erd\"os, Ginzburg, and Ziv proved the 
following classical theorem.

\begin{theorem}[Erd\"os--Ginzburg--Ziv \cite{erdos1961theorem}]
\label{thm:egz}
Any sequence of length $2n - 1$ in $\ZZ/n\ZZ$ contains
a zero-sum subsequence of length $n$.
\end{theorem}

Here, a subsequence need not be consecutive, and
a sequence is \emph{zero-sum} if its elements sum to
$0$. This theorem has lead to many 
problems involving zero-sum sequences over groups.

In general, let $G$ be an abelian group, and
let $G_0 \subseteq G$ be a susbset. Let $\mathcal L
\subseteq \NN$. Then $s_{\mathcal L}(G_0)$ is defined
to be the minimal $\ell$ such that any sequence
of length $\ell$ with elements in $G_0$ contains
a zero-sum subsequence whose length is in $\mathcal L$.
When $G_0 = G$ and $\mathcal L = \{\exp(G)\}$,
this constant is called the
Erd\"os--Ginzburg--Ziv constant.

When $G = \ZZ$, this problem turns out to be not
very interesting --- if $G_0$ contains a nonzero element,
then $s_{\mathcal L}(G_0) = \infty$.
This has lead to \cite{augspurger2016avoiding} the study of
the \emph{modified Erd\"os--Ginzburg--Ziv constant}
$s_{\mathcal L}'(G_0)$, defined as the smallest
$\ell$ such that any \emph{zero-sum} sequence of length at least
$\ell$ with elements in $G_0$ contains a zero-sum
subsequence whose length is in $\mathcal L$.
When $\mathcal L = \{t\}$ is a single element,
we omit the set brackets for convenience.
In \cite{berger2016analogue}, the first author
determined modified EGZ constants in the infinite
cyclic case.
Here we treat the finite cyclic case and extensions.

\begin{prob}[{\cite[Problem 2]{berger2016analogue}}]
\label{prob:problem}
	Compute $s_t'(G)$ for $G = \ZZ/n\ZZ$ and $\Znn$.
\end{prob}

In this paper, we answer Problem \ref{prob:problem} 
for $G = \ZZ/n\ZZ$ and for $t = n$, $G = \Znn$.
Note that in both cases, when $n$ does not divide $t$, 
the quantity $s_t(G)$ is infinite.

\begin{restatable}{thmm}{cyclic}
\label{thm:cyclic}
	The modified EGZ constant of $\ZZ/n\ZZ$ is given by $s_{nt}'(\ZZ/n\ZZ)
	= (t+1)n - \ell + 1$, where $\ell$ is the smallest integer
	such that $\ell \nmid n$.
\end{restatable}

\begin{restatable}{thmm}{square}
\label{thm:square}
	We have $s_n'(\Znn) = 4n - \ell + 1$ where $\ell$ is
	the smallest integer such that $d \ge 4$ and
	$\ell \nmid n$.
\end{restatable}

\section{The cyclic case}
In this section we give the proof of Theorem \ref{thm:cyclic}.
As in \cite{reiher2007kemnitz}, if $J$ is a sequence of
elements of $\ZZ/n\ZZ$ or $\Znn$, we use
$(k \mid J)$ to denote the number of zero-sum subsequences of
$J$ of size $k$.

\begin{proposition}
\label{prop:cyclicupper}
	If $d \mid n$ and $J$ is a zero-sum sequence in
	$\ZZ/n\ZZ$ of length $2n - d$, then $(n \mid J) > 0$.
\end{proposition}
\begin{proof}
	By Theorem \ref{thm:egz}, we can break off 
	subsequences of $J$ of size $d$ with sum $0 \pmod d$
	until we have fewer than $2d-1$ remaining.
	In fact, since $d \mid n$, we will have exactly
	$d$ remaining. But since the
	sum was zero-sum to begin with, the last $d$ must
	also sum to zero, so we have $2(n/d) - 1$
	blocks of size $d$ with sums
	$dx_1, \dots, dx_{2(n/d) - 1}$ for some $x_i$.
	By Theorem \ref{thm:egz}, some $n/d$ of these
	must sum to $0$ in $\ZZ/(n/d)\ZZ$, so the union
	of these blocks gives a subsequence of length
	$n$ whose sum is zero in $\ZZ/n\ZZ$.
\end{proof}

\begin{corollary}
\label{cor:cyclicupper}
	Let $\ell$ be the smallest positive integer such that
	$\ell \nmid n$, and let $t \ge 1$. If $J$ is
	a zero-sum sequence in $\ZZ/n\ZZ$ of length at least
	$(t + 1)n - \ell + 1$, then $(nt \mid J) > 0$.
\end{corollary}
\begin{proof}
We induct on $t$. The case $t = 1$ follows from Proposition
	\ref{prop:cyclicupper} since
	$\ell - 1, \dots, 1$ all divide $n$. Suppose the result
	is true for positive integers less than $t > 1$.
	Then $J$ contains a zero-sum subsequence of length
	$(t-1)n$. Remove these elements from $J$.
	We are left with a zero-sum sequence of length 
	$2n - \ell + 1$. This is the $t = 1$ case, so we can
	find another zero-sum subsequence of length $n$.
	Combine this with the $(t-1)n$ to get the desired
	subsequence of length $nt$.
\end{proof}

\begin{proposition}
\label{prop:constructionn}
	Suppose $\ell \nmid n$ and $t \ge 1$. Then there exists a zero-sum
	subsequence in $\ZZ/n\ZZ$ of length $(1 + t)n - \ell$ which
	contains no zero-sum subsequence of length $nt$.
\end{proposition}
\begin{proof}
	Consider a sequence of
	$0$'s and $1$'s with multiplicities $a \le tn - 1$,
	$b \le n - 1$ respectively where
	$a + b = (t+1)n - \ell$. Such a sequence will have
	no zero-sum subsequence of length $nt$.
	It suffices to find $a$, $b$ such that
	$g = \gcd(n, \ell) \mid b$, because then we can
	add some constant to every term of the sequence to make
	it zero-sum. Note that adding a constant to
	every term does not introduce any new zero-sum
	subsequences. It suffices to take
	$b = tn - g$ and $a = n - \ell + g
	\le n - \ell/2 \le n - 1$.
\end{proof}

Corollary \ref{cor:cyclicupper} and Proposition \ref{prop:constructionn} together imply
Theorem \ref{thm:cyclic}.

\section{The case $\Znn$}
In this section we prove Theorem \ref{thm:square}.
We first prove some preliminary lemmas.
The following results from \cite{reiher2007kemnitz}
are key.

\begin{lemma}[{\cite[Corollary 2.4]{reiher2007kemnitz}}]
\label{lem:por2p}
	Let $p$ be a prime, and let $J$ be a sequence of 
	elements in $(\ZZ/p\ZZ)^2$. If $|J| = 3p-2$ or
	$|J| = 3p-1$, then $(p \mid J) = 0$ implies
	$(2p \mid J) \equiv - 1 \pmod p$.
\end{lemma}

\begin{lemma}[{\cite[Corollary 2.5]{reiher2007kemnitz}}]
\label{lem:3p}
	Let $p$ and $J$ be as in Lemma \ref{lem:por2p}.
	If $|J|$ is a zero-sum sequence with exactly $3p$
	elements, then $(p \mid J) > 0$.
\end{lemma}

\begin{theorem}[{\cite[Theorem 3.2]{reiher2007kemnitz}}]
\label{thm:egznn}
	If $J$ is a sequence of length $4n -3$ in 
	$\Znn$ then $(n \mid J) > 0$.
\end{theorem}

We generalize Lemma \ref{lem:3p} to non-prime $n$.

\begin{lemma}
\label{lem:3n}
	If $J$ is a zero-sum sequence of length $3n$
	in $(\ZZ/n\ZZ)^2$, then $(n \mid J) > 0$.
\end{lemma}
\begin{proof}
	We induct on $n$. The base case $n = 1$ is clear.
	Assume the the lemma is true for all positive integers
	less than $n$. Let $n = pm$	with $p$ prime and $m < n$.

	Since $3n > 4m - 3$, we can find some $m$ elements
	of $J$ whose sum is $0 \pmod m$. Say their sum
	is $mx_1$ and remove these $m$ elements. We can continue
	doing this until there remain only $3m$ elements.
	But since $J$ was a zero-sum sequence, the remaining
	$3m$ elements must sum to $0 \pmod m$, so by the
	induction hypothesis, we can remove another $m$
	with sum a multiple of $m$. This gives us
	$3p-2$ blocks of size $m$ whose sums are
	$mx_1, \dots, mx_{3p-2}$ for some $x_i$.

	If some $p$ of the $x_i$ sum to $0 \pmod p$, then
	combining the blocks would give us $n$ elements whose
	sum is $0 \pmod n$, as desired. If not, by Lemma
	\ref{lem:por2p}, we must have some $2p$ of the $x_i$
	summing to $0 \pmod p$, so we have $2n$ elements
	whose sum is $0 \pmod n$. But since $J$ itself is
	zero-sum and has size $3n$, the complement
	is zero-sum as well and has size $n$.
\end{proof}

\begin{proposition}
\label{prop:upperbound}
	If $d\mid n$, and $J$ is a zero-sum sequence
	in $\Znn$ of length $4n - d$, then $(n \mid J ) > 0$. 
\end{proposition}
\begin{proof}
	Note that $4n - d \ge 3m$. By Theorem
	\ref{thm:egznn}, we can break off subsequences of
	size $d$ with sum $0 \pmod d$ until we have only
	$3d$ elements remaining. Then by Lemma \ref{lem:3n} we can
	break off another $d$ elements, to obtain
	$4(n/d) - 3$ blocks of size $d$,
	with sums $dx_1, \dots, dx_{4(n/d) - 3}$
	for some $x_i$.
	By Theorem \ref{thm:egznn}, some $n/d$ of the $x_i$
	must sum to $0$ in $(\ZZ/(n/d)\ZZ)^2$.
	Combining the corresponding blocks gives a subsequence
	of length $n$ whose sum is zero in
	$\Znn$.
\end{proof}

The following corollary is clear from Proposition
\ref{prop:upperbound} and Theorem \ref{thm:egznn}.

\begin{corollary}
\label{cor:upperbound}
	Let $\ell$ be the smallest integer greater than or
	equal to $4$ such that
	$\ell \nmid n$.
	If $J$ is a zero-sum sequence in $\Znn$ of length at least	
	$4n - \ell + 1$, then $(n \mid  J) > 0$.
\end{corollary}

\begin{proposition}
\label{prop:constructionnn}
	Suppose $4 \le \ell \nmid n$. There exists
	a zero-sum sequence in $\Znn$ of length $4n - \ell$
	which contains no zero-sum subsequences of length $n$.
\end{proposition}
\begin{proof}
	First, consider a sequence of the form
	\begin{align*}
		(0,0)\quad & a \le n - 1 \\
		(0,1)\quad & b \le n - 1\\
		(1,0)\quad & c \le n - 1\\
		(1,1)\quad & d \le n - 1,
	\end{align*}
	where $a$ denotes the number of $(0,0)$'s, etc.,
	and $a + b + c + d = 4n - \ell$.
	It is easy to check that this sequence contains no
	zero-sum subsequence of length $n$.
	Now, we claim that there exists
	$(r,s) \in \Znn$ such that adding $(r,s)$ to each
	term of the above sequence will result in a zero-sum
	sequence. Note that adding $(r,s)$ to each term
	does not change the fact that there is no zero-sum
	subsequence of length $n$.

	In fact, all we need is 
	\[
		g \coloneqq \gcd(n, \ell) \mid c + d,
		b + d.
	\]
	We claim that the following	$a, b, c, d$ work.
	\begin{align*}
		a &= n - \ell + g + 1 \text{ (or } n - \ell + 2g + 1
		\text{ if } g = 1) \\
		b &= n - 1\\
		c &= n - 1 \\ 
		d &= n - g + 1 \text{ (or } n - 2g + 1 \text{ if } g = 1) 
		.
	\end{align*}
	Note that $g \le \ell/2$ because $\ell \nmid n$, so
	$a \le n - \ell/2 + 1 \le n - 1$ if $g \neq 1$,
	and $a = n - \ell + 3 \le n - 1$ if $g = 1$.
	It is easy to show that we always have $a, d \ge 0$
	and $d \le n - 1$, and that these $a, b, c, d$
	satisfy the divisibility relation.
\end{proof}

Now, Corollary \ref{cor:upperbound} and Proposition
\ref{prop:constructionnn} imply Theorem \ref{thm:square}.

\section{Open problems}
Harborth \cite{harborth1973ein} first considered the problem of computing
$s_n((\ZZ/n\ZZ)^d)$ for higher dimensions. He proved 
the following bounds.

\begin{theorem}[Harborth \cite{harborth1973ein}]
\label{thm:harborth}
	We have
	\[
		(n-1)2^d + 1 \le s_n((\ZZ/n\ZZ)^d) \le (n-1)n^d + 1.
	\]
\end{theorem}

For $d > 2$ the precise value of $s_n((\ZZ/n\ZZ)^d)$
is not known. See \cite{edel2007zero,elsholtz2004lower}
for some better lower bounds and
\cite{alon1993zero,meshulam1995subsets} for some
better upper bounds. In general the lower bound 
in Theorem \ref{thm:harborth} is not tight,
but Harborth showed that it is an equality
for $n = 2^k$ a power of $2$. 

\begin{conjecture}
	If $n = 2^k$ and $d \ge 1$, we have
	\[
		s_n'(\Znd) = 2^d n - \ell + 1,
	\]
	where $\ell$ is the smallest integer such that
	$\ell \ge 2^d$ and $\ell \nmid n$.
\end{conjecture}
By an argument similar to the $\Znn$ case, we can
reduce this conjecture to the case $n = 2^d$,
in which case $\ell = 2^d + 1$. We also have not 
determined the modified EGZ constants for $\Znn$
for subseqences of length greater than $n$.

\begin{prob}
	Compute $s_{nt}'(\Znn)$ for $t > 1$.
\end{prob}

The constant $s_n(\Zmn)$ is known to be $2m + 2n - 3$
for $m \mid n$ \cite[Theorem 5.8.3]{halter2006non}. 
\begin{prob}
	Compute $s_{nt}'(\Zmn)$ for $t \ge 1$ and $m \mid n$.
\end{prob}

\section{Acknowledgements}
This research was conducted at the University of Minnesota
Duluth REU and was supported by NSF / DMS grant 1650947 and 
NSA grant H98230-18-1-0010.
We would like to thank Joe Gallian for running the program.

\bibliography{duluth}
\bibliographystyle{acm}
\end{document}